\documentclass[11pt]{article}
\usepackage [dvips]{graphics}
\usepackage[centertags]{amsmath}
\usepackage{amsfonts}
\usepackage{amssymb}
\usepackage{amsthm}
\usepackage{newlfont}
\usepackage{stmaryrd}
\usepackage{mathrsfs}
\usepackage[colorlinks,
            linkcolor=blue,
            anchorcolor=blue,
            citecolor=blue
            ]{hyperref}

\theoremstyle{plain}
\newtheorem{thm}{Theorem}[section]
\newtheorem{cor}[thm]{Corollary}
\newtheorem{lem}[thm]{Lemma}

\newtheorem{rem}[thm]{Remark}
\newtheorem{defi}[thm]{Definition}

\def\sqr#1#2{{\vcenter{\vbox{\hrule height.#2pt
              \hbox{\vrule width.#2pt height#1pt \kern#1pt \vrule
width.#2pt}
              \hrule height.#2pt}}}}
\def\dbN{{\mathbb{N}}}

\def\dbR{{\mathbb{R}}}

\def\r{\rho}
\def\a{\alpha}

\def\l{\lambda}

\def\3n{\negthinspace \negthinspace \negthinspace }
\def\2n{\negthinspace \negthinspace }
\def\1n{\negthinspace }

\def\O{\Omega}
\def\no{\noindent}

\def\ms{\medskip}
\def\bs{\bigskip}

\def\({\Big (}
\def\){\Big )}
\def\[{\Big[}
\def\]{\Big]}

\def\be{\begin{equation}}
\def\bel{\begin{equation}\label}
\def\ee{\end{equation}}
\def\bea{\begin{eqnarray}}
\def\eea{\end{eqnarray}}
\def\bt{\begin{theorem}}
\def\et{\end{theorem}}
\def\bc{\begin{corollary}}
\def\ec{\end{corollary}}
\def\bl{\begin{lemma}}
\def\el{\end{lemma}}
\def\bp{\begin{proposition}}
\def\ep{\end{proposition}}
\def\br{\begin{remark}}
\def\er{\end{remark}}
\def\ba{\begin{array}}
\def\ea{\end{array}}
\def\bd{\begin{definition}}
\def\ed{\end{definition}}

\makeatletter
   
   \@addtoreset{equation}{section}
\makeatother

\begin{document}

\title{ \bf A note on ``Problem of eigenvalues of stochastic Hamiltonian systems with boundary conditions"}

\author{ Guangdong Jing\thanks{ Partially supported by NSFC (No.11871308), E-mail: jingguangdong@mail.sdu.edu.cn},
\quad Penghui Wang \thanks{Partially supported by NSFC (No.11471189, No.11871308), {
E-mail:} {phwang@sdu.edu.cn}. \ms}
 \\ \\
 School of Mathematics, Shandong University,    Jinan,  250100,  China\\
}

\maketitle
\begin{abstract}
The eigenvalue problem of stochastic Hamiltonian systems with boundary conditions was studied by Peng \cite{peng} in 2000. For one-dimensional case, denoting by $\{\lambda_n\}_{n=1}^{\infty}$  all the eigenvalues of such an eigenvalue problem, Peng proved that $\lambda_n\to +\infty$. In this short note, we prove that the growth order of $\lambda_n$ is the same as $n^2$ as $n\to +\infty$. Apart from the interesting of its own, by this result, the statistic period  of solutions of FBSDEs can be estimated directly by corresponding coefficients and time duration.
\end{abstract}

\bs

\no{\bf 2000 MSC}.  34L15, 60H10

\bs

\no{\bf Key Words}. Eigenvalue problem; Stochastic Hamiltonian system, Forward-Backward Stochastic Differential Equations

\section{Introduction and main results}
Let $(\Omega,\mathscr{F}, \mathbb{F}, \mathbb{P})$
be a complete filtered probability space, on which a standard one-dimensional
Brownian motion $B=\{B_t\}_{t\ge0}$ is defined, and
$\mathbb{F}=\{\mathscr{F}_t\}_{t\geq 0}$ is the natural filtration of $B$
augmented by all the $\mathbb{P}$-null sets in $\mathscr{F}$.
Let $T>0$ be any fixed time horizon.

In \cite{peng}, Peng considered the following eigenvalue problem of stochastic Hamiltonian system with boundary conditions:
\begin{equation}  \label{evp2}
\left\{
\begin{aligned}
& \mathrm{d}x_t=[H_{21}^\l x_t+ H_{22}^\l y_t
   +H_{23}^\l z_t]\mathrm{d}t+[H_{31}^\l x_t+H_{32}^\l y_t
   +H_{33}^\l z_t]\mathrm{d}B_t,\indent t\in[0,T],                \\
& -\mathrm{d}y_t=[H_{11}^\l x_t+H_{12}^\l y_t
   +H_{13}^\l z_t]\mathrm{d}_t-z_t \mathrm{d} B_t,  \indent t\in[0,T],  \\
& x(0)=0,\indent y(T)=0,
\end{aligned}
\right.
\end{equation}
where $H^\l=H-\l \bar{H}$,
$$
  H=\begin {bmatrix}
  H_{11}&H_{12}&H_{13}\\
  H_{21}&H_{22}&H_{23}\\
  H_{31}&H_{32}&H_{33}
  \end{bmatrix},      \indent
  \bar{H}=\begin {bmatrix}
  \bar{H}_{11}&\bar{H}_{12}&\bar{H}_{13}\\
  \bar{H}_{21}&\bar{H}_{22}&\bar{H}_{23}\\
  \bar{H}_{31}&\bar{H}_{32}&\bar{H}_{33}
  \end{bmatrix},
$$
$H^\l_{ij}=H_{ij}-\l \bar{H}_{ij}$, $H_{ij}=H_{ij}^T$,
$\bar{H}_{ij}=\bar{H}_{ij}^T$, $i,j=1,2,3$, which are constant matrices.
\begin{defi}
A real number $\lambda$ is called an \emph{eigenvalue} of linear stochastic Hamiltonian system with boundary conditions \eqref{evp2} if there exists a nontrivial solution $(x,y,z)$ of \eqref{evp2}.
This  solution is called an \emph{eigenfunction} corresponding to $\lambda$.
All eigenfunctions associated with the eigenvalue $\lambda$ constitute a linear subspace of $M^2(0,T;\dbR^n)$, called the \emph{eigenfunction subspace} corresponding to $\lambda$.
\end{defi}

The above eigenvalue problem is a stochastic analogue of classical
eigenvalue problem of mechanic systems, and it is closely related
to the existence  of solutions to Forward-Backward Stochastic Differential
Equations (FBSDEs in short). Please refer to \cite{hupeng, maprotteryong, MWZZ,peng, wupeng}
and references therein for the well-posedness of FBSDEs,
among which the \emph{monotonicity condition} is an important sufficient condition for the existence and uniqueness of solutions to FBSDEs.
In particular, for linear FBSDE \eqref{evp2} with $\bar{H}=0$, the monotonicity condition  has the following concrete form:
\begin{eqnarray}\begin {bmatrix}
-H_{11}&-H_{12}&-H_{13}\\
H_{21}&H_{22}&H_{23}\\
H_{31}&H_{32}&H_{33}
\end{bmatrix}
\leqslant-\alpha I_{3n},\label{moncondition}\end{eqnarray}
where $\alpha>0$ is a constant.

For the case
\begin{equation}\label{pert-speci-H22}
\bar{H}=
\begin{bmatrix}
  0&0&0\\
  0&H_{22}&0\\
  0&0&0
\end{bmatrix},
\end{equation}
by using the blow-up time (as usual, the \emph{blow-up time} denotes the endpoints of the maximum existing interval of the solution to certain ODEs) of the solution for the related Riccati equation,
Peng proved the following
\begin{thm}\cite[Theorem 3.2]{peng} \label{thmPeng}
For \eqref{evp2} of one dimension with perturbation \eqref{pert-speci-H22}, assume that  \eqref{moncondition}
is satisfied as well as $H_{23}=-H_{33}H_{13}$.
Then there exist $\{\lambda_n\}_{n=1}^{+\infty}$, all the eigenvalues of
the problem (\ref{evp2}), such that $\lambda_n\rightarrow+\infty$ as \ $n\rightarrow +\infty$.
Moreover, the eigenfunction space corresponding to each $\lambda_n$ is of  one dimension.
\end{thm}

The existence of eigenvalues is given in the above theorem. Then it is natural and meaningful to ask that whether those eigenvalues have any relationship with the coefficients of systems and how they tend to infinity.
Towards solving these problems we have the following

\begin{thm}  \label{thm1}
Under the same assumptions in Theorem \ref{thmPeng},
$$
\lambda_n= O(n^2),  \indent  \text{as}\ \  n\to+\infty.
$$
In detail,
\begin{equation*}
\frac{\pi^2}{-2H_{11}H_{22}T^2} \le \mathop{\underline{\lim}}_{n\to+\infty} \frac{\lambda_n}{n^2}\leq \mathop{\overline{\lim}}_{n\to+\infty}  \frac{\lambda_n}{n^2}\le \frac{4\pi^2}{-H_{11}H_{22}T^2}.
\end{equation*}
\end{thm}
\begin{rem}
The result about order in Theorem \ref{thm1} can be considered as an analogue of the well-known result in deterministic case.
\end{rem}


In literature, the study of FBSDEs is mainly focused on the existence and uniqueness of solutions and hardly any on properties of solutions. From the point of view of eigenvalue problem of stochastic Hamiltonian system with boundary conditions, what is different is that some concrete characteristic such as statistic periodicity and stochastic oscillations of solutions of FBSDEs can be given.
Remark that by the proof of \cite[Theorem 3.2]{peng}, the serial number of eigenvalues are already related to the statistic periodicity of its corresponding eigenfunctions. But both of them are yet isolated from the coefficients of the systems.
It is worth noting that by the main result in this paper,  the statistic period of the solutions of the FBSDEs can be estimated directly by its coefficients and time duration.

\begin{cor}
Let $\lambda$ be an eigenvalue of the stochastic Hamiltonian system in Theorem \ref{thm1}, for sufficiently large $n$, if
\begin{equation*}
  \lambda<\frac{n^2 \pi^2}{-2H_{11}H_{22}T^2},\indent \left(\text{resp.} \quad \lambda>\frac{4n^2\pi^2}{-H_{11}H_{22}T^2}, \right)
\end{equation*}
the statistic period  of the associate eigenfunctions (i.e., the solutions of FBSDEs) is less (resp. greater) than $n$.
\end{cor}


The rest of the paper is organized as follows.
In Section \ref{section2}, we recall some preliminary results and  give several lemmata.
The proof of Theorem \ref{thm1} is given in Section \ref{section3}.

\section{Preliminaries and several lemmata}\label{section2}
For one-dimensional case with perturbation \eqref{pert-speci-H22}, the eigenvalue problem of stochastic Hamiltonian system with boundary conditions \eqref{evp2} is rewritten as
\begin{equation}  \label{eqn1}
\left\{
\begin{aligned}
&\mathrm{d}x_t=[H_{21}x_t+(1-\l) H_{22}y_t+H_{23}z_t]\mathrm{d}t   \\
& \indent\ \ \ +[H_{31}x_t+H_{32}y_t+H_{33}z_t]\mathrm{d}B_t,    \indent t\in[0,T],       \\
&-\mathrm{d}y_t=[H_{11}x_t+H_{12}y_t
  +H_{13}z_t]\mathrm{d}t-z_t \mathrm{d}B_t,  \indent t\in[0,T],              \\
&x(0)=0,\indent y(T) =0.
\end{aligned}
\right.
\end{equation}
As given in \cite[Subsection 4.2]{peng},
through Legendre transformation,
the dual Hamiltonian $\tilde{H}$ of the original Hamiltonian $H$ corresponding
to (\ref{eqn1}) is
$$\tilde{H}=
\begin {bmatrix}
H_{33}^{-1}H_{32}^2-\rho H_{22}&H_{33}^{-1}H_{32}H_{31}-H_{21}&-H_{33}^{-1}H_{32}\\
H_{33}^{-1}H_{32}H_{31}-H_{21}&H_{33}^{-1}H_{31}^2&-H_{33}^{-1}H_{31}\\
-H_{33}^{-1}H_{32}&-H_{33}^{-1}H_{31}&H_{33}^{-1}
\end{bmatrix},$$
where $\rho={1-\lambda}$,
and the relation between solution $(x,y,z)$ of original Hamiltonian system and solution $(\tilde{x},\tilde{y},\tilde{z})$ of dual Hamiltonian system is:
\begin{equation*}
\left\{
\begin{aligned}
&x(t)=\tilde{y}(t), \indent  y(t)=\tilde{x}(t),    \\
&z(t)=-H_{33}^{-1}H_{32}\tilde{x}(t)-H_{33}^{-1}H_{31}\tilde{y}(t)
+H_{33}^{-1}\tilde{z}(t).
\end{aligned}
\right.
\end{equation*}

In Peng \cite{peng}, the idea to study the eigenvalue problem of stochastic
Hamiltonian system is to deal with the blow-up time of the following
Riccati equations with terminal conditions \cite[(6.2)]{peng}:
\begin{equation}\label{r11}
\left\{
\begin{aligned}
&\frac {\mathrm{d}k}{\mathrm{d}t}=-(2H_{21}+H_{13}^2)k
-H_{11}-(\rho H_{22}-H_{33}H_{13}^2)k^2, \indent  t\le T,  \\
& k(T)=0,
\end{aligned}
\right.
\end{equation}
and dual Riccati equations with terminal conditions \cite[(6.4)]{peng}:
\begin{equation}\label{r21}
\left\{
\begin{aligned}
&\frac{\mathrm{d}\tilde{k}}{\mathrm{d}t}
=(2H_{21}+H_{13}^2)\tilde{k}+H_{11}\tilde{k}^2
+(\rho H_{22}-H_{33}H_{13}^2),    \indent  t\le T,    \\
&\tilde{k}(T)=0.
\end{aligned}
\right.
\end{equation}

The following two lemmata from \cite{peng} are needed.
\begin{lem}\cite[Lemma 6.1]{peng}  \label{lem-peng1}
For the Riccati equation (\ref{r11}), the blow-up time $t_\rho$ is
continuous and strictly decreasing with respect to $\rho$ when\ $\rho<\rho_0$, where $\rho_0=H_{22}^{-1}H_{33}H_{13}^2$.
Besides,
\begin{equation}
\lim_{\rho\to-\infty}t_\rho=T,\indent
\lim_{\rho\to\rho_0^-}t_\rho=-\infty.
\end{equation}
\end{lem}
\begin{lem} \cite[Lemma 6.2]{peng}   \label{lem-peng2}
For the dual Riccati equation (\ref{r21}), the blow-up time
$\tilde{t}_\rho$ is continuous and strictly decreasing with
respect to $\rho$ when \ $\rho<\rho_0$.
Besides,
\begin{equation}
\lim_{\rho\to-\infty}\tilde{t}_\rho=T,\indent
\lim_{\rho\to\rho_0^-}\tilde{t}_\rho=-\infty.
\end{equation}
\end{lem}
To simplify the notation, denote
\begin{equation}    \label{rln-p-q-r}
\left\{
\begin{aligned}
&q(\rho)=-(\rho H_{22}-H_{33}H_{13}^2),\indent  r=-H_{11},     \\
&\tilde{r}(\rho)=(\rho H_{22}-H_{33}H_{13}^2),\indent
   \tilde{q}=H_{11},                                        \\
&\tilde{p}=-p=2H_{21}+H_{13}^2,                            \\
&\rho_*=(4H_{11}H_{22})^{-1}{(2H_{21}+H_{13}^2)^2}.
\end{aligned}
\right.
\end{equation}
%
By  (\ref{moncondition}),
$$r=-H_{11}<0,\indent H_{22}<0,\indent  H_{33}<0.$$
Besides, by \cite[Page 278]{peng}, $\rho_0$ is the critical point,
which implies that all the eigenvalues of problem \eqref{eqn1} are located in $(-\infty, \rho_0)$.
Moreover, by Theorem \ref{thmPeng}, there are at most finite eigenvalues $\{\lambda_i\}_{i=1}^{m}$ of problem \eqref{eqn1}, such that  $1-\lambda_i \in [\rho_0+\rho_*, \rho_0)$.
Since we merely try to depict the growth order of $\lambda_n$ in this paper, it is reasonable for us to only check those $\rho\in(-\infty, \rho_0+\rho_*)$.
Actually, that
$\rho<\rho_0+\rho_*$
guarantees
\begin{eqnarray} \label{nontrivial-condition-delta}
1-\frac{p^2}{4rq(\rho)}=1-\frac{\tilde{p}^2}{4\tilde{r}(\rho)\tilde{q}}>0.
\end{eqnarray}

In what follows, for simplicity, we sometimes omit the $\rho$ in $q(\rho)$ and $\tilde{r}(\rho)$.


The following two lemmata depict the solutions of \eqref{r11} and \eqref{r21}, which are essential
in proving Theorem \ref{thm1}.


\begin{lem}     \label{lem1}
The blow-up time\ $t_\rho$ of solution $k$ of  \eqref{r11} satisfies
\begin{equation}  \label{b-utk1}
\sqrt{rq(\rho)- \frac{p^2}{4}}(T-t_\rho)+\arctan \frac{-p}{\sqrt{4rq(\rho)-p^2}}
=\frac{\pi}{2}.
\end{equation}
\begin{proof}
Under \eqref{nontrivial-condition-delta}, $pk+r+qk^2<0$ and $1-\frac{p^2}{4rq }>0$.
Then
\begin{equation*}
\frac{\mathrm{d}k}{\frac{p}{r}k+\frac{q}{r}k^2+1}=\frac{\mathrm{d}k}
{\left(\sqrt \frac{q}{r}k-\frac{p}{2\sqrt{rq}}\right)^2+\left(1-\frac{p^2}{4rq}\right)}
=r\mathrm{d}t.
\end{equation*}
Combined with terminal condition $k(T)=0$,
\begin{equation*}
k=-\frac{\sqrt{4rq-p^2}}{2q}\tan\left[\sqrt{rq-\frac{p^2}{4}}(T-t)
+\arctan\frac{-p}{\sqrt{4rq-p^2}}\right]-\frac{p}{2q}.
\end{equation*}
Therefore, for any fixed $\rho\ (\rho<\rho_0+\rho_*)$,
the blow-up time $t_\rho$ of $k$ satisfies
\begin{equation*}
\sqrt{rq- \frac{p^2}{4}}(T-t_\rho)+\arctan \frac{-p}{\sqrt{4rq-p^2}}
=\frac{\pi}{2}.
\end{equation*}
\end{proof}
\end{lem}
\begin{lem}    \label{lem2}
The blow-up time $\tilde{t}_\rho$ of solution $\tilde{k}$ of   (\ref{r21})  satisfies
\begin{equation}   \label{b-utk2}
\sqrt{\tilde{q}\tilde{r}(\rho)-\frac{\tilde{p}^2}{4}}(\tilde{t}_\rho-T)
+\arctan{\frac{\tilde{p}}{\sqrt{4\tilde{q}\tilde{r}(\rho)-\tilde{p}^2}}}
=-\frac{\pi}{2}.
\end{equation}
\begin{proof}
Under \eqref{nontrivial-condition-delta},
$\tilde{r}+\tilde{p}\tilde{k}+\tilde{q}\tilde{k}^2 >0$
and $1-\frac{\tilde{p}^2}{4\tilde{r} \tilde{q}}>0$.
Then
\begin{equation*}
\frac{\mathrm{d}\tilde{k}}{1+\frac{\tilde{p}}{\tilde{r}}\tilde{k}
+\frac{\tilde{q}}{\tilde{r}}\tilde{k}^2}
=\frac{\mathrm{d}\tilde{k}}{\left(\sqrt{\frac{\tilde{q}}{\tilde{r}}}\tilde{k}
+\frac{\tilde{p}}{2\sqrt{\tilde{q}\tilde{r}}}\right)^2
+\left(1-\frac{\tilde{p}^2}{4\tilde{q}\tilde{r}}\right)}=\tilde{r}\mathrm{d}t.
\end{equation*}
Combined with terminal condition $\tilde{k}(T)=0$,
\begin{equation*}
\tilde{k}=\frac{\sqrt{4\tilde{q}\tilde{r}-\tilde{p}^2}}
{2\tilde{q}}\tan\left[\frac{\sqrt{4\tilde{q}\tilde{r}-\tilde{p}^2}}{2}(t-T)+
\arctan{\frac{\tilde{p}}{\sqrt{4\tilde{q}\tilde{r}-\tilde{p}^2}}}\right]
-\frac{\tilde{p}}{2\tilde{q}}.
\end{equation*}
For any fixed $\rho\ (\rho<\rho_0+\r_*)$,
the blow-up time $\tilde{t}_\rho$ of solution $\tilde{k}$
satisfies
\begin{equation*}
\sqrt{\tilde{q}\tilde{r}-\frac{\tilde{p}^2}{4}}(\tilde{t}_\rho-T)
+\arctan{\frac{\tilde{p}}{\sqrt{4\tilde{q}\tilde{r}-\tilde{p}^2}}}
=-\frac{\pi}{2}.
\end{equation*}
\end{proof}
\end{lem}

\section{Proof of Theorem \ref{thm1}}\label{section3}
In this section, we  prove Theorem \ref{thm1}.
\begin{proof}[Proof of Theorem \ref{thm1}]
By \cite[Section 6, Proof of Theorem 3.2.]{peng},
the\ $n$-th eigenvalue $\lambda_n$ of
(\ref{eqn1}) is uniquely determined by
\begin{equation*}
t^{2n-1}_{\rho_n}=0,\indent \lambda_n=1-\rho_n,\indent  \forall n \in \dbN_+,
\end{equation*}
where
\begin{equation*}
t^{2j-1}_{\rho_n}=T-j(T-t_{\rho_n})-(j-1)(T-\tilde{t}_{\rho_n}),
  \indent j=1,2,\cdots n,
\end{equation*}
\begin{equation*}
t^{2j-2}_{\rho_n}=T-(j-1)(T-t_{\rho_n})-(j-1)(T-\tilde{t}_{\rho_n}),
  \indent t^1_{\rho_1}=t_\rho.
\end{equation*}
By Lemma \ref{lem1} and Lemma \ref{lem2}, we obtain
\begin{equation*}
\sqrt{rq- \frac{p^2}{4}}(T-t_\rho)+\arctan \frac{-p}{\sqrt{4rq-p^2}}
=\frac{\pi}{2}
\end{equation*}
and
\begin{equation*}
\sqrt{\tilde{q}\tilde{r}-\frac{\tilde{p}^2}{4}}(\tilde{t}_\rho-T)
+\arctan{\frac{\tilde{p}}{\sqrt{4\tilde{q}\tilde{r}-\tilde{p}^2}}}
=-\frac{\pi}{2}.
\end{equation*}
Then
$$\lim\limits_{n\to+\infty}\frac{T-t_{\rho_n}}{T-\tilde{t}_{\rho_n}} = 1$$
and
\begin{equation}\label{twoequationlambda}
\sqrt{rq- \frac{p^2}{4}}=\frac{\frac{\pi}{2}+\arctan{\frac{p}{\sqrt{4rq-p^2}}}}{T-t_\rho}
=\frac{\frac{\pi}{2}-\arctan{\frac{p}{\sqrt{4rq-p^2}}}}{T-\tilde{t}_\rho}.
\end{equation}
Besides, following the method in \cite[Subsection 6.2]{peng},
\begin{equation}\label{eq6}
T-n(T-t_{\rho_n})-(n-1)(T-\tilde{t}_{\rho_n})=0.
\end{equation}
Then
\begin{equation*}
  \max\left\{ T-t_{\rho_n}, T-\tilde{t}_{\rho_n} \right\} \ge \frac{T}{2n-1}
\end{equation*}
and for sufficiently large $n$,
\begin{equation*}
  \max\left\{ T-t_{\rho_n}, T-\tilde{t}_{\rho_n} \right\} \le \frac{T}{2n-2}.
\end{equation*}
Then by \eqref{twoequationlambda}, for sufficiently large $n$,
\begin{equation*}
 \frac{(n-1)\pi}{T} \le \sqrt{rq(\rho_n)- \frac{p^2}{4}} \le \frac{(2n-1)\pi}{T}.
\end{equation*}
Moreover, by \eqref{rln-p-q-r},
\begin{equation*}
rq(\rho_n)-\frac{p^2}{4}=\left(-H_{11}H_{22}\right)\lambda_n
+H_{11}H_{22}-H_{11}H_{33}H_{13}^2-\frac{(2H_{21}+H_{13}^2)^2}{4}.
\end{equation*}
Then
\begin{equation*}
\frac{\pi^2}{-2H_{11}H_{22}T^2} \le \mathop{\underline{\lim}}_{n\to+\infty} \frac{\lambda_n}{n^2}\leq \mathop{\overline{\lim}}_{n\to+\infty}  \frac{\lambda_n}{n^2}\le \frac{4\pi^2}{-H_{11}H_{22}T^2}.
\end{equation*}
\end{proof}





\end{document}